\documentclass[12pt,twoside]{article}
\usepackage[T1]{fontenc}
\usepackage[utf8]{inputenc}
\usepackage{amsmath,amsthm,amssymb}
\usepackage{times}
\usepackage{enumitem}
\usepackage{amsfonts,calrsfs,color,verbatim,eucal,yfonts,mathrsfs,mathtools,bbm}
\usepackage{hyperref}
\usepackage{bbm}

\def\titlerunning#1{\gdef\titrun{#1}}
\makeatletter
\def\author#1{\gdef\autrun{\def\and{\unskip, }#1}\gdef\@author{#1}}

\makeatother

\newtheorem{Theorem}{Theorem}[section]
\newtheorem{Definition}[Theorem]{Definition}
\newtheorem{Proposition}[Theorem]{Proposition}

\newtheorem{Remark}[Theorem]{Remark}
\newtheorem{Example}[Theorem]{Example}

\numberwithin{equation}{section}

\def\R{\mathbb R}
\def\N{\mathbb N}

\def\KC{\mathcal K}
\def\BC{\mathcal B}

\def\MC{\mathcal M}
\def\FC{\mathcal F}

\def\<{\langle}
\def\>{\rangle}

\newcommand{\yx}{\color{black}}

\newcommand{\xy}{\color{black}}

\begin{document}


\baselineskip=17pt


\titlerunning{Markov uniqueness and Fokker-Planck-Kolmogorov equations}

\title{{\Large \bf Markov uniqueness and Fokker-Planck-Kolmogorov equations}}

\author{{Sergio Albeverio$^{a}$},~~{Vladimir I. Bogachev$^{b}$},~~{Michael R\"{o}ckner$^{c, d}$}
\\
 \small $a.$ Institute for Applied Mathematics and HCM, University of Bonn, 53115 Bonn, Germany \\
 \small $b.$ Lomonosov Moscow State University, Russia and National Research \\\small
 University Higher
School of Economics, Russian Federation\\
  \small $c.$ Faculty of Mathematics, Bielefeld University, 33615 Bielefeld, Germany\\
   \small $d.$ Academy of Mathematics and Systems Science, CAS, 100190 Beijing, China}
\date{}
\maketitle
\begin{center}
\large{Dedicated to Masatoshi Fukushima for his $88^{th}$ birthday.}
\end{center}
\normalsize\begin{minipage}{145mm}
{\bf Abstract.}
In this paper we show that Markov uniqueness for symmetric pre-Dirichlet operators $L$ follows from 
the uniqueness of the corresponding Fokker-Planck-Kolmogorov equation (FPKE). Since in recent years a 
considerable number of uniqueness results for FPKE's have been achieved, we obtain new Markov uniqueness results 
in concrete cases. A selection of such will be presented in this paper. They include cases with killing and with degenerate diffusion coefficients.

\end{minipage}

\section{Introduction and framework}\label{section:1}

In this paper we fix a $\sigma$-finite measure space $(E,\BC,m)$. Let
$L^p:=L^p(m)=L^p(E,m)$, $p\in[1,\infty]$ be
the corresponding (real) $L^p$-spaces
 with their usual norms $\|\cdot\|_p$ and inner product $(\;,\;)_2$
 if $p=2$. On $L^p(m)$ we shall consider linear operators
 $$
 L\colon D(L)\subset L^p(m)\to L^p(m)
 $$ 
 with their usual partial order defined by
\begin{align*}
L_1\subset L_2 \underset{Def.}{\Leftrightarrow} \Gamma(L_1)\subset\Gamma(L_2),
\end{align*}
where $D(L)$ is a linear subspace of $L^p$, called domain of $L$, and
\begin{align*}
\Gamma(L):=\{(u,Lu)\in L^p\times L^p\colon u\in D(L) \}
\end{align*}
is the graph of $L$. In particular, we shall consider those $L$
 which generate a (unique) strongly continuous semigroup of (everywhere defined)
 continuous linear operators on $L^p$, denoted by $e^{tL}$, $t\geq0$.
 Henceforth such $L$ will be shortly called generator (on $L^{P}$).
 We refer to \cite{Pa1985} for the notions and the well-known characterization of such generators. We recall that a generator $L$ is always closed, i.e. $\Gamma(L)$ is a closed subset of $L^p\times L^p$, with domain $D(L)$ dense in $L^p$ and that for $p\in(1,\infty)$, its adjoint operator $(L^*,D(L^*))$ on $L^{p'}$, with $p':=\frac{p}{p-1}$, generates a strongly continuous semigroup of linear operators, $e^{tL^*}$, $t\geq0$, on $L^{p'}$. This satisfies
\begin{align}\label{adjoint}
e^{tL^*}=(e^{tL})^*,\quad t\geq0,
\end{align}
(see \cite[Chapter~1, Corollary 10.6]{Pa1985}).
We consider three cases of sets of generators on $L^p$ for $p\in(1,\infty)$:
\begin{enumerate}
\item[(1)] Let $D^*_0$ be a dense linear subspace of $L^{p'}$
and $L^*_0:D_0^*\subset L^{p'}\to L^{p'}$ a linear operator.
Define $\MC:=\MC(L_0^*,D_0^*)$ to be the set of all linear operators
$L\colon D(L)\subset L^p\to L^p$ such that $L_0^*\subset L^*$ and $L$ is a generator on $L^{p}$.

\item[(2)] Let $D_0$ be a dense linear subspace of $L^2$ and $L_0\colon
D_0\subset L^2\to L^2$ a symmetric linear operator, i.e., $L_0\subset L_0^*$, which is upper bounded, i.e.
\begin{align*} 
\sup_{u\in D_0\setminus\{0\}}(L_0u,u)_2 \yx ||u||_2^{-2} \xy \yx<\xy\yx\infty\xy.
\end{align*}
 Define $\MC_{sym}:=\MC_{sym}(L_0,D_0)$ to be the set of all linear
 operators $L\colon D(L)\subset L^2\to L^2$ such that $L_0\subset L$,
 $L$ is a generator on $L^{2}$ and $L$ is symmetric, i.e., $L\subset L^*$.

\item[(3)] Let $(L_0,D_0)$ be as in (2) and define $\MC_{sym,M}:=\MC_{sym,M}(L_0,D_0)$
to be the subset of all $(L,D(L))$ in $\MC_{sym}(L_0,D_0)$ such that each $e^{tL}$, $t\geq0$,
 is sub-Markovian, i.e., if $u\in L^2$ such that $0\leq u\leq 1$, then $0\leq e^{tL}u\leq1$.
\end{enumerate}

Concerning (2) we note that by \cite[Theorems 4.2 and 5.3]{Pa1985} it obviously follows
that $\MC_{sym}(L_0,D_0)$ coincides with the set of all linear operators
 $L\colon D(L)\subset L^2\to L^2$ such that $L_0\subset L$ and $L$ is upper
 bounded and self-adjoint, i.e., $L=L^*$. Furthermore, $\MC_{sym}(L_0,D_0)$ is not empty,
 because the Friedrichs extension of $(L_0,D_0)$ is self-adjoint and upper bounded
 (see e.g. \cite[p.~131]{FOT2011}).

Concerning (3) we refer to \cite{FOT2011} and \cite{MR1992} for more details on such
sub-Markovian operator semigroups.

The first aim of this paper is to derive a "parabolic" condition in each
of the cases (1),(2),(3) which implies that the respective sets $\MC,\MC_{sym},\MC_{sym,M}$
contain at most one element. Here, "parabolic" means in terms of
the corresponding Fokker-Planck-Kolmogorov equation (FPKE).
The second aim of this paper is (by refining this "parabolic condition")
to use uniqueness results from \cite{BKRS2015} to obtain new results on
"Markov uniqueness" \yx in the sense of the following definition:\xy

\begin{Definition}
\yx Let $(L_0,D_0)$ be as in (2) above. $(L_0,D_0)$ is called Markov unique if $\MC_{sym,M}$ contains  exactly one element.\xy
\end{Definition}

Let us note that our notion of "Markov uniqueness" is in fact stronger than
the one extensively studied in the literature, since there, uniqueness
is studied in the subset of all linear operators $(L,D(L))$ in $\MC_{sym,M}(L_0,D_0)$,
which are nonpositive definite, i.e., $\sup_{u\in D(L)}(Lu,u)_2\leq0$,
while also assuming that $(L_0,D_0)$ is nonpositive definite.

The literature on Markov uniqueness is \yx quite extensive and a number of types of state spaces $E$, as e.g. $\R^d$ or infinite dimensional vector spaces 
or manifolds have been considered .\xy To the best of our knowledge
the first paper on this subject is \cite{Ta1985} by Masayoshi Takeda.
To give an overview of the entire literature is beyond the scope of this paper.
Instead, we refer to the references in \cite{FOT2011},
\cite{Eb1999}, \yx \cite{EL2006}, \cite{RS2011}, \xy \cite{AR1995}, \cite{ARZ1993} and the more recent paper\yx s  \cite{AMR2014}\xy, \cite{RZZ2017}.
\\
It seems, however, that the method to prove Markov uniqueness proposed in this paper,
i.e., by using the corresponding FPKE, is new, though it is very natural.
Furthermore our applications and examples in Section 3\yx ,\ even though they are all in the classical case $E:=\R^d$, \xy appear to be not
covered by the existing literature, in particular, since \yx they include cases with degenerate diffusion coefficients and \xy we can allow "killing",
more precisely in our applications, where $L$ is a partial differential operator
on $\R^d$, this operator is allowed to have a (negative) zero order coefficient.
\\ \\
\yx Finally, we would like to recall the notion of "strong uniqueness" which is different from Markov uniqueness. In our context here
it means that the larger set $\MC_{sym} (L_0,D_0)$ contains exactly one element which is equivalent to the fact that the closure of $(L_0,D_0)$ is self adjoint on $L^2$.
For more details we refer to \cite{AKR1995}, \cite{Eb1999} and as a very recent paper to \cite{AKMR2020}, in particular to the lists of references in them.
\xy

\section{The main idea and a parabolic condition for uniqueness}\label{section:2}

For a set $\FC$ of real-valued functions on $E$ and $T\in(0,\infty)$ we define
$\FC_{,T}$ to be the set of all functions of the form
\begin{align*}
[0,T]\times E \ni (t,x)\mapsto f(t)\varphi(x)=:(f\otimes\varphi)(t,x),
\end{align*}
where $\varphi\in\FC$ and $f\in C^1([0,T];\R)$ with $f(T)=0$.

Let us start with case (1) from the introduction and consider $(L,D(L))\in \MC(L^*_0,D_0^*)$.

For $t\geq0$ we set
\begin{align*}
T_t^L:=e^{tL},\ T_t^{L^*}=e^{tL^*}
\end{align*}
(cf. \eqref{adjoint}). Then for all $\varphi\in D_0^*$, $u\in L^p$ and $t\geq0$ we have
\begin{align}\label{eq:2.1}
\int \varphi\ T_t^{L}u\ dm &= \int T_t^{L^*} \hspace{-0.4em}\varphi\ u\ dm \notag\\
&= \int \varphi\ u \ dm + \int_0^t\int T_s^{L^*} \hspace{-0.1cm}L^{*}\varphi\ u\ dm\ ds\notag\\
&= \int\varphi\ u\ dm + \int_0^t\int L_0^*\varphi\ T_s^{L}u\ dm\ ds.
\end{align}
Hence defining the (signed) measure $\mu_t(dx):=T_t^L\yx u\xy(x)\ m(dx),\ t\geq0$,
by the (integral) product rule for all $f\otimes\varphi\in D_{0,T}^*$ \yx (defined as above with $\FC := D_0^*$) \xy we have

\begin{align*}
\int(f\otimes\varphi)(t,x)\ \mu_t(dx)=f(t)\int\varphi(x)\ \mu_t(dx)
\end{align*}

\begin{align*}
&=f(0) \int\varphi\ d\mu_0+\int_0^t f(s)\int L_0^*\varphi\ d\mu_s\ ds
+\int_0^t f'(s)\int\varphi\ d\mu_s\ ds\\
&=\int (f\otimes\varphi)(0,x)\ \mu_0(dx)+\int_0^t\int(\frac{\partial}{\partial s}+L_0^*)(f\otimes\varphi)\ d\mu_s\ ds.
\end{align*}
In particular, for $t=T$
\begin{align}\label{eq:2.2}
\int_0^T \int (\frac{\partial}{\partial s}+L_0^*)(f\otimes\varphi)\ d\mu_s\ ds=-\int(f\otimes\varphi)(0,x)\ \mu_0(dx).
\end{align}
\eqref{eq:2.1} (equivalently \eqref{eq:2.2}) means that $\mu_t= T_t^{\yx L\xy} \yx u\xy\cdot m, \yx u\in L^p, \xy\ t\geq0$, solves the FPKE \yx (up to time $T$ for every $T\in (0,\infty))$
\xy
corresponding to $(L_0^*,D(L_0^*))$ (see \cite{BKRS2015}).

Now it is very easy to prove the following ''parabolic condition'' that ensures
that $$\#\MC(L_0^*,D_0^*)\leq1 $$ \yx (where as usual $\#$ is an abbreviation for cardinality).  \xy

\begin{Proposition}\label{prop2.1}
Assume that for every $T\in(0,\infty)$
\begin{align}\label{2.3}
(\frac{\partial}{\partial s}+L_0^*)D_{0,T}^*\text{ is dense in }L^{p'}([0,T]\times E,dt\otimes m).
\end{align}
Then $\MC(L_0^*,D_0^*)$ contains at most one element.
\end{Proposition}
\begin{proof}
Let $(\tilde{L},D(\tilde{L}))\in\MC_{sym}(L_0^*,D_0^*)$. Then,
as seen above, $\tilde{\mu}_t:=T_t^{\tilde{L}}\yx u\xy\ m$, \yx $u \in L^p,$ \xy $t\geq0$,
also satisfies \eqref{eq:2.1}, hence \eqref{eq:2.2}. So, (by subtracting)
for $g(t,\cdot):=T_t^L\yx u\xy-T_t^{\tilde{L}}\yx u\xy,\ t\geq0,$ we obtain for all $T\in(0,\infty)$
\begin{align*}
\int_0^T\int(\frac{\partial}{\partial s}+L_0^*)(f\otimes\varphi)\ g(s,\cdot)\ dm\ ds = 0
\end{align*}
for all $f\otimes\varphi\in D_{0,T}^*$.
Since $g\in L^p([0,T]\times E,\ dt\otimes m)$, by \eqref{2.3} it follows that $g=0$, and the assertion follows, since $\yx u\xy \in L^p$ was arbitrary.
\end{proof}

Now let us consider case (2) from the introduction. So, let $(L,D(L))\in\MC_{sym}(L_0,D_0)$. Then using the same notation 
as in case (1) we analogously obtain for all $\varphi\in D_0,\ u\in L^2$ and $t\geq0$
\begin{align}\label{eq:2.4}
\int \varphi\ T_t^{L}u\ dm = \int \varphi\ u\ dm + \int_0^t\int L_0\varphi\ T_s^{L}u\ dm\ ds,
\end{align}
hence for $\mu_t:=T_t^{L}u\ m$ and for all $f\otimes\varphi\in D_{0,T}$, $T\in(0,\infty)$
we have
\begin{align}\label{eq:2.5}
\int_0^T(\frac{\partial}{\partial s}+L_0^*)(f\otimes\varphi)\ d\mu_s\ ds=-\int(f\otimes\varphi)(0,x)\ \mu_0(dx),
\end{align}
i.e., $\mu_t,\ t\geq0,$ solves the FPKE corresponding to $(L_0,D_0)$.\\
Analogously to Proposition \ref{prop2.1} we then prove the following result.

\begin{Proposition}\label{prop2.2}
Assume that for every $T\in(0,\infty)$
\begin{align}\label{2.6}
(\frac{\partial}{\partial s}+L_0)D_{0,T} \text{ is dense in } L^2([0,T]\times E,dt\otimes m).
\end{align}
Then $\MC_{sym}(L_0,D_0)$ consists of exactly one element.
\\
\
\\
\textnormal{In case (3) if $(L,D(L))\in\MC_{sym,M}(L_0,D_0)$, then obviously $T_t^L(L^2\cap L^\infty)\subset L^2\cap L^\infty$, and since $L^2\cap L^\infty$ is dense in $L^2$, $T_t^L$ is uniquely determined on $L^2\cap L^\infty$.
}
\end{Proposition}

\begin{Proposition}\label{prop2.3}
\textnormal{
Suppose that $L_0(D_0) \subset L^{1}$ (which automatically holds if $m(E) < \infty$) and that for every T $\in (0,\infty)$}
\begin{align}\label{eq:2.7}
(\frac{\partial}{\partial s}+L_0)D_{0,T} \text{ is dense in } L^1([0,T]\times E,dt\otimes m).
\end{align}
Then $\MC_{sym, M}(L_0,D_0)$ consists of at most one element. If the semigroup generated by the Friedrichs 
extension of $(L_0, D_0)$ is sub-Markovian, then this extension is the unique element in $\MC_{sym,M}(L_0,D_0)$.
\end{Proposition}
\begin{proof}
We repeat the proof of Proposition \ref{prop2.2} (respectively, \ref{prop2.1}) with $u \in L^2 \cap L^{\infty}$. 
Then for all $T \in (0, \infty)$ we have for $g$ as in the proof of Proposition 2.3 that $g \in L^{\infty} ([0, T] \times E,dt\otimes m)$. 
Then by \eqref{eq:2.7} we conclude again that $g = 0$. Since $T^{L}_{t}$ is uniquely determined on $L^{2} \cap L^{\infty}$, the assertion follows.
\end{proof}

Clearly, conditions \eqref{2.3}, \eqref{2.6} and \eqref{eq:2.7} are not easy to check in applications 
and certainly too strong, at least in case (3). So, let us discuss a weaker condition in this case.

Let $(L, D(L)) \in \MC_{sym, M}(L_{0}, D_{0})$ and fix $u \in L^{\infty}$ such that $u \geq 0 $ and $\int u\; dm = 1$. 
Then $\mu^{L}_{t} := T^{L}_{t} u\; m,\; t \geq 0$, are subprobability measures on $(E, \BC)$, i.e., $\mu^{L}_{t}(E) \leq 1$ 
for all $t \geq 0$. We note that obviously each $T^{L}_{t}$ is uniquely determined by its values on such $u$. 
We have seen that $\mu_t := \mu^{L}_{t}, t \geq 0$, solves the corresponding FPKE

\begin{align}\label{eq:2.9}
\int\varphi \ d\mu_{t} = \int\varphi\ d\mu_{0} + \int_0^t \int L_{0} \varphi \ d\mu_{s} ds, t \geq 0, \forall \varphi \in D_{0},
\end{align}
hence for all $T \in (0, \infty), f \otimes \varphi \in D_{0, T}$
\begin{align}\label{eq:2.10}
\int_0^T \int (\frac{\partial}{\partial s}+L_0)(f\otimes\varphi)\ d\mu_s\ ds = - \int (f \otimes \varphi)(0, x)\mu_{0}(dx).
\end{align}
\eqref{eq:2.9} and \eqref{eq:2.10} are equations for paths $(\mu_{t})_{t \geq 0}$ of subprobability measures on $(E, \BC)$ 
such that $[0, \infty) \ni t \to \mu_{t}(A)$ is Lebesgue measurable for all $ A \in \BC$. Define $\mathcal S\mathcal P$ to be the set of all such paths, \yx and $\mathcal S\mathcal P$(T) the set of their restrictions to [0,T] \xy

Now the following result is obvious.

\begin{Theorem}\label{thm2.4}
If, for every probability density $u \in L^{1} \cap L^{\infty}$, \eqref{eq:2.9} (or \eqref{eq:2.10}) 
has a unique solution $(\mu_{t})_{t \geq 0} \in \mathcal S\mathcal P$ such that each $\mu_t, t \geq 0$, 
is absolutely continuous with respect to $m$ and such that $\mu_0 = u \cdot m$, then $\MC_{sym, M}(L_0, D_0)$ 
consists of at most one element. If, in addition, the semigroup generated by the Friedrichs extension of $(L_0, D(L_0))$ is sub-Makovian,
then this extension is the unique element in $\MC_{sym, M}(L_0, D_0)$.
\end{Theorem}

In Theorem \ref{thm2.4}, it is enough to prove uniqueness for \eqref{eq:2.10} (or \eqref{eq:2.11}) 
in the subclass of all $(\mu_t)_{t \geq 0} \in \mathcal S\mathcal P$ for which each $\mu_t, t \geq 0$, is absolutely continuous 
with respect to $m$ with bounded density, i.e., one only needs uniqueness in a convex subset of $\mathcal S\mathcal P$. 
Therefore, the following result, which was first observed in \cite{BDPRSt2007}, is useful and goes beyond absolutely continuous solutions.

\begin{Proposition}\label{prop2.4}
Let \yx T $\in (0,\infty)$ and \xy $\zeta$ be a subprobability measure on $(E, \BC)$ and let $\KC_{\zeta} \subset \mathcal S\mathcal P$\yx(T) \xy
be a non-empty convex set such that each $(\mu)_{t \yx \in  [0,T]\xy} \in \KC_{\zeta}$ is a solution to \eqref{eq:2.9} 
(hence to \eqref{eq:2.10}) \yx with \xy $\mu_{0} = \zeta$. Suppose that for every $(\mu_{t})_{\yx t \in [0,T]\xy} \yx\in \KC_{\zeta}\xy $
\begin{align}\label{eq:2.11}
(\frac{\partial}{\partial s}+L_0)(D_{0, T}) \text{ is dense in } L^{1}([0, T] \times E, \mu_{t}dt) .
\end{align}
Let $(\mu_{t})_{\yx t \in [0,T] \xy}, (\tilde{\mu}_{t})_{\yx t \in [0,T] \xy} \in \KC_{\zeta}$.
 Then $\mu_{t} = \tilde{\mu}_{t}$ for $dt$-a.e. $t \in [0, \yx T \xy]$.
\end{Proposition}

\begin{Remark}\rm
Clearly for $(L, D(L)) \in \MC_{sym, M}(L_{0}, D_{0})$ and the corresponding solutions $(\mu_{t}^{L})_{t \geq 0}$ defined above, 
condition \eqref{eq:2.11} is weaker than condition \eqref{eq:2.7} in Proposition \ref{prop2.3}, 
since $\sup_{t\in [0, T]} \Vert{T^{L}_{t}u}\Vert_{\infty} < \infty$.
\end{Remark}

\begin{proof}[Proof of Proposition 2.5]
Since $\KC_{\zeta}$ is convex, we have that $\nu_{t} := \frac{1}{2} \mu_{t} + \frac{1}{2} \tilde{\mu}_{t}, t \geq 0$, 
is again in $\KC_{\zeta}$ and
\begin{align}
\mu_t dt = g \nu_t dt,\;  \tilde{\mu}_t dt = \tilde{g} \nu_t dt
\end{align}
for some $g, \tilde{g} \in L^{\infty} ([0, T] \times E, \nu_{t} dt)$.

Furthermore, by \eqref{eq:2.10} it follows that for all $f \otimes \varphi \in D_{0, T}$
\begin{align}
\int_0^T \int (\frac{\partial}{\partial s}+L_0)
(f\otimes\varphi)(g - \tilde{g})\ d\nu_s\ ds = 0.
\end{align}
Hence by (2.10) this implies that $g = \tilde{g}$ and the assertion follows.
\end{proof}

 Proposition \ref{prop2.4} and the observation that (at least in many cases)
  it suffices to check \eqref{eq:2.11} for just one solution in $\KC_{\zeta}$,
  are the core of the proof of many results on uniqueness
  of solutions in $\mathcal S\mathcal P$ to concrete FPKEs
  in Chapter 9 of \cite{BKRS2015},
  which thus can be applied to prove Markov uniqueness
  for many examples of given operators $(L_{0}, D_{0})$ on $L^{2}(m)$ as above.
  We shall present a selection of such in the next section.
  We shall restrict ourselves to the symmetric case, i.e.,
  $p = 2$ and $L_{0} \subset L_{0}^*$, though also nonsymmetric cases
  (as in case (1) from Section 1) can be treated
   if one has enough knowledge about the dual operator
   $(L_{0}^*, D_{0}^*)$ on $(L^{p})'$ for $p \in (1, \infty)$ (see Remark 4.2 below).

\section{Some uniqueness results for FPKEs}\label{section:3}

\yx In the rest of the paper we shall concentrate on the case where the state space $E$ is equal to $\R^d$. By the same ideas it is, however, possible to 
obtain Markov uniqueness from uniqueness results of FPKEs on more general state spaces, including infinite dimensional vector spaces or manifolds. This will be done in future work.
\xy
\subsection{Fokker-Planck-Kolmogorov equations}

As already mentioned we shall use the uniqueness results on FPKEs
from \cite[Chapter 9]{BKRS2015}.
So, let us briefly recall the framework there, but for simplicity
restricting to solutions in $\mathcal S\mathcal P$, since we shall
only use these in our applications below.

Below $(E,\BC)$ from the previous sections will always be $(\R^d,\BC(\R^d)),\ d\in\N,$
where $\BC(\R^d)$ denotes the Borel $\sigma$-algebra of $\R^d$.
Consider a partial differential operator of the form
\begin{align}\label{eq:3.1}
L_{0}\varphi=a^{ij}\partial_{x_i}\partial_{x_j}\varphi+b^i\partial_{x_i}\varphi+c\varphi,\quad \varphi\in D_0:=C_0^\infty(\R^d),
\end{align}
where we use Einstein's summation convention,
$\partial_{x_i}:=\frac{\partial}{\partial x_i},\ 1\leq
i\leq d,\ a^{ij},\ b^i,\ c\ \colon [0,T]\times\R^d\to\R$,
with $c\leq0$, are $\BC(\R^d)$-measurable functions,
$A(t,x):=(a^{ij}(t,x))_{1\leq i,j\leq d}$ is a nonnegative definite
matrix for all $(t,x)\in[0,T]\times\R^d$ and $T\in(0,\infty)$ is fixed.
For some of the results below we need to assume local boundedness
and local strict ellipticity of $A$, i.e.\yx :\xy
\begin{enumerate}
  \item[(H1)] For each ball $U\subset\R^d$
  there exist $\gamma(U),\ M(U)\in(0,\infty)$ such that
  \begin{align*}
 \gamma(U)\cdot I \leq A(t,x)\leq M(U)\cdot I \quad \forall\ (t,x)\in[0,T]\times\R^d,
  \end{align*}
  where $I$ denotes the $d\times d$ identity matrix.
\end{enumerate}

Let $\mathcal S\mathcal P$ be defined as in Section \ref{section:2}.
We say that $(\mu_t)_{t\geq0}\in\mathcal S\mathcal P$ satisfies the
FPKE (up to time $T$ for $L_0$)
 if $a^{ij},\ b^{i},\ c\in L^1_{loc}([0,T]\times\R^d,\mu_t\ dt)$ and
 for every $\varphi\in C^\infty_0(\R^d)$
\begin{align}\label{eq:3.2}
\int\varphi\ d\mu_t=\int \varphi\ d\mu_0
+\int_0^t\int L_0\varphi\ d\mu_s\ ds\quad \text{for $dt$-a.e. }t\in[0,T].
\end{align}
In Subsection \ref{eq:3.2} - \ref{eq:3.4} below we shall only be interested in the
so-called subprobability solutions to \eqref{eq:3.2}, i.e., we a priori
restrict to a class $\mathcal S\mathcal P_{\nu}\subset \mathcal S\mathcal P$
in which we search for a (hopefully unique) solution to \eqref{eq:3.2}.
So, given a subprobability measure $\nu$ on $\BC(\R^d)$ (i.e.,
 $\nu\geq0$ and $\nu(\R^d)\leq1)$, $\mathcal S\mathcal P_\nu$ is defined to
 be the set of all $(\mu_t)_{ \yx t \in [0,T] \xy}\in\mathcal S\mathcal P\yx(T)\xy$ with the following properties:
\begin{align}
&(\mu_t)_{\yx t\in [0,T] \xy} \text{ solves \eqref{eq:3.2}},\label{eq:3.3}\\
&c\in L^1([0,T]\times \R^d,\mu_t\ dt),\label{eq:3.4}\\
&b\in L^2([0,T]\times U, \mu_t\ dt;\R^d) \text{ for all balls }U\subset\R^d,\label{eq:3.5}\\
\mu_0=\nu \text{ and } &\mu_t(\R^d)\leq\nu(\R^d)+\int_0^t\int c(x,s)\ \mu_s(dx)\ ds \text{ for $dt$-a.e. }t\in[0,T]\label{eq:3.6}.
\end{align}
Clearly, if $\nu\neq0$, by dividing by $\nu(\R^d)$, we may assume, without loss of generality concerning the 
uniqueness of solutions in $\mathcal S\mathcal P_\nu$ for \eqref{eq:3.2}, that $\nu(\R^d)=1$.
Below we fix a probability measure $\nu$ on $\BC(\R^d)$.
\\
\
\\
Now let us recall several uniqueness results for \eqref{eq:3.2} from \cite[Chapter 9,]{BKRS2015}. Below let $dx$ denote Lebesgue measure on $\R^d$.
\subsection{Nondegenerate VMO diffusion coefficients}

Let us recall the definition of the VMO(=vanishing mean oscillation)-property of a function (see \cite{K2007} 
and the references therein), which is a vast generalization of local Lipschitzianity.

Let $g$ be a bounded \yx Borel-measurable \xy function on $\R^{d+1}$. Set
\begin{align*}
O(g,R) :=\sup_{(x,t)\in\R^{d+1}}&\sup_{r\leq R}r^{-2}|U(x,r)|^{-2}\times\\
&\times\int_t^{t+r^2}\int\int_{y,z\in U(x,r)}|g(s,y)-g(s,z)|\ dy\ dz\ ds.
\end{align*}
If $\lim\limits_{R\to0}O(g,R)=0$, then we say that the function $g$ belongs to the class $VMO_x(\R^{d+1})$.

Suppose that a \yx Borel-measurable \xy function $g$ is defined on $\yx[0,T]\times\R^d\xy$ and bounded on $\yx[0,T]\times U\xy$ for every ball $U$. 
We extend $g$ by zero to the whole space $\R^{d+1}$. If for every function $\zeta\in C_0^\infty(R^d)$ the 
function $g\zeta$ belongs to the class $VMO_x(\R^{d+1})$, then we say that $g$ belongs to the class $VMO_{x,loc}([0,T]\times \R^d)$.

\begin{Theorem}\label{thm3.1}
Let (H1) hold and assume that
\begin{align*}
a^{ij}\in VMO_{x,loc}([0,T]\times\R^d),\ 1\leq i,j\leq d.
\end{align*}
Then the set
\begin{align}\label{eq:3.7}
\{(\mu_t)_{\yx t \in [0,T]\xy}\in \mathcal S\mathcal P_\nu\ :\ a^{ij},\ b^{i}\in L^1([0,T]\times\R^d,\mu_t\ dt) \}
\end{align}
contains at most one element.
\end{Theorem}
\begin{proof}
See \cite[Theorem 9.3.6]{BKRS2015}.
\end{proof}

\subsection{Nondegenerate locally Lipschitz diffusion coefficients}

In this subsection and the next one we use the following condition:
\begin{enumerate}
  \item[(H2)] For every ball $U\subset\R^d$ there exists $\Lambda(U)\in(0,\infty)$ such that for all $1\leq i,j\leq d$
  \begin{align*}
  |a^{ij}(t,x)-a^{ij}{t,y}|\leq\Lambda(U)|x-y|\quad \forall t\in[0,T],\ x,y\in U.
  \end{align*}
\end{enumerate}

\begin{Theorem}\label{thm3.2}
Suppose that conditions (H1) and (H2) hold, that $c\leq0$ and that 
$b\in L^p_{loc}([0,T]\times\R^d, dt\ dx;\R^d)$, $c\in L^{\frac{p}{2}}_{loc}([0,T]\times\R^d,dt\ dx)$ 
for some $p>d+2$. Assume also that there exists $(\mu_t)_{t\geq0}\in \mathcal S\mathcal P_\nu$ satisfying the condition
\begin{align*}
|a^{ij}|/(1+|x|^2)+|b^{i}|/(1+|x|)\in L^1([0,T]\times\R^d, \mu_{t} dt), 1 \leq i, j \geq d.
\end{align*}
Then the set $\mathcal S\mathcal P_\nu$ consists of exactly one element.
\end{Theorem}
\begin{proof}
See \cite[Theorem 9.4.3]{BKRS2015}.
\end{proof}

\subsection{Nondegenerate diffusion coefficients and the Lyapunov function condition}

The function $V$ in the following theorem is called a Lyapunov function.

\begin{Theorem}\label{thm3.3}
Suppose that conditions (H1) and (H2) hold, $c\leq0$ and that $b\in L^p_{loc}([0,T]\times\R^d, dt\ dx;\R^d)$, 
$c\in L^{\frac{p}{2}}_{loc}([0,T]\times\R^d,dt\ dx)$ for some $p>d+2$. Suppose also that there exists a positive 
function $V\in C^2(\R^d)$ such that $V(x)\to+\infty$ as $|x|\to+\infty$ and for some $C \in (0, \infty)$ and all $(t, x)\in[0,T]\times\R^d$ we have
\begin{align*}
L_{0}V(t,x)\leq C+CV(x).
\end{align*}
Then the set $\mathcal S\mathcal P_\nu$ contains at most one element.
\end{Theorem}
\begin{proof}
See \cite[Theorem 9.4.6]{BKRS2015}.
\end{proof}

\begin{Example}\label{example3.4}
Let $V(x)=\ln(|x|^2+1)$ if $|x|>1$. Then the condition $L_{0}V\leq C+CV$ is equivalent to the inequality
\begin{align}\label{eq: 3.8}
2\ trA(t,x)&-4\frac{\langle A(t,x)x,x\rangle}{|x|^2+1}+c(t,x)(|x|^2+1)\ln(|x|^2+1)+2\langle b(t,x),x\rangle\\
&\leq C(|x|^2+1)+C(|x|^2+1)ln(|x|^2+1)\notag.
\end{align}
\end{Example}

\begin{proof}
See \cite[Theorem 9.4.7]{BKRS2015}.
\end{proof}

\subsection{Degenerate diffusion coefficients}

\subsubsection{A uniqueness result of LeBris/Lions}
Here we assume that $c = 0$ in \eqref{eq:3.1}, i.e., we consider a partial  differential operator of the form
\begin{align}
L_{0}\varphi=a^{ij}\partial_{x_i}\partial_{x_j}\varphi+b^i\partial_{x_i}\varphi, \quad \varphi\in D_0:=C_0^\infty(\R^d),
\end{align}
where $a^{ij}, b^{i}, 1 \leq i, j\leq d$, are as in \eqref{eq:3.1}, and its corresponding FPKE (\ref{eq:3.2}).

Let $\sigma^{ij}\colon [0, T] \times \R^{d} \to \R$ be $\BC (\R^{d})$-measurable functions such that $A = \sigma \sigma^{*}$,
 where $\sigma := (\sigma^{ij})_{1 \geq i, j \geq d}$. Set 
 $$
 \beta^i := b^i - \partial_{x_j} a^{ij}, \ 1 \leq i, j \leq d.
$$
The following result is due to C. LeBris and P.L. Lions (see \cite[Proposition 5]{LBL2008}, and also \cite[Theorem 9.8.1]{BKRS2015}).

\begin{Theorem}\label{thm3.5}
Suppose that in the natural notation
\begin{align*}
\sigma^{ij} \in L^{2} ([0, T]; W_{loc}^{1, 2}(\R^{d}, dx)), \quad \beta^{i} \in L^{1} ([0, T]; W_{loc}^{1, 1}(\R^{d}, dx)),
\end{align*}
\begin{align*}
div\; \beta \in L^{1}([0, T]; L^{\infty}(\R^{d}, dx)),\; \; \frac{|\beta|}{1 + |x|} \in L^{1}([0, T]; L^{1}(\R^{d}, dx)) + L^{1} ([0, T]; L^{\infty}(\R^{d}, dx)),
\end{align*}
\begin{align*}
\frac{\sigma^{ij}}{1 + |x|} \in L^{2}([0, T];\;  L^{2}(\R^{d}, dx)) + L^{2} ([0, T]; L^{\infty}(\R^{d}, dx)).
\end{align*}
Then, for every initial condition given by density $\rho_{0}$ from $L^{1}(\R^{d}, dx) \cap L^{\infty} (\R^{d}, dx)$ 
there exists a unique solution to \eqref{eq:3.2} with $\mu_0 := \rho_{0} dx$ in the class
\begin{align*}
\{
 \rho\colon \rho \in L^{\infty} ([0, T]; L^{1}(\R^{d}, dx) \cap L^{\infty}(\R^{d}, dx)), \sigma^{*} \nabla \rho \in L^{2} ([0, T]; L^{2}(\R^{d}, dx)).
\}
\end{align*}
\end{Theorem}

\subsubsection{Uniqueness in the class of absolutely continuous paths of probability measures}

Here we assume
\begin{enumerate}
  \item[(H3)] (H1) is satisfied with $\gamma=\gamma(U),\ M=M(U)$, independent of the ball $U$ and $(t,x)\mapsto A(t,x)$ is 
  Lipschitz in $t$ and $x$ on $[0,T]\times\R^d$, \yx $T > 0$. \xy
  \item[(H4)] $b\in L^\infty([0,\infty\yx)\xy\times\R^d,dt\ dx;\R^d)$.
\end{enumerate}
Furthermore, we fix a $\BC([0,T]\times\R^d)$-measurable non-negative function $\tilde\rho\colon [0,\yx \infty)\xy\times\R^d\to[0,\infty).$

Consider the operator
\begin{align}\label{eq:3.8}
L_0 \varphi=\tilde\rho\ div(A\nabla \varphi)+\sqrt{\tilde\rho}\langle b,\nabla \varphi\rangle,\ \varphi\in D_0:=C_0^\infty(\R^d).
\end{align}
and its corresponding FPKE \eqref{eq:3.2}.

Define $\mathcal Z_\nu$ to be the set of all $(\mu_t)_{\yx t \in [0,T]\xy}\in\mathcal S\mathcal P\yx(T)\xy$ such that $\mu_0=\nu$ \yx and  
$\mu_t\ $dt is absolutely continuous w.r.t. dxdt with density \xy $z:=\frac{d(\mu_t\ dt)}{dx\ dt}$ \yx satisfying \xy the following properties:
\begin{align}
&(\mu_t)_{\yx t\in [0,T] \xy} \text{ solves the FPKE corresponding to \eqref{eq:3.8}.}\label{eq:3.9}\\
&\mu_t(\R^d)=1 \text{ for $dt$-a.e. }t\in[0,T].\label{eq:3.10}\\
&\tilde\rho z\in L^2([0,T]\times U, dt\ dx) \text{ for all balls }U\subset \R^d.\label{eq:3.11}\\
\lim_{N\to\infty}\int_0^T&\int_{N\leq|x|\leq2N}\left[\frac{\sqrt{\tilde\rho(t,x)}+\tilde\rho(t,x)}{1+|x|}z(t,x)
+\frac{\tilde\rho^2(t,x)}{1+|x|^2}z^2(t,x)\right]dx\ dt=0. \label{eq:3.12}
\end{align}

\begin{Theorem}\label{thm3.5b}
Suppose that (H3) and (H4) hold. Then $\mathcal Z_\nu$ contains at most one element.
\end{Theorem}
\begin{proof}
This follows from \cite[Theorem 9.8.2]{BKRS2015}.
\end{proof}

\section{Applications to the Markov uniqueness problem}

\subsection{The Framework}

Also in this section we take $(E, \BC) := (\R^{d}, \BC(\R^{d}))$ and $m := \rho\; dx$, where $$
\rho \in L^{1}_{loc} (\R^{d}, dx),
\ \rho > 0 \; \hbox{$dx$-a.e.}
$$
 We consider the following partial differential operator:
\begin{align}\label{eq:4.1}
L_0 \varphi = \frac{1}{\rho}\partial_{x_{i}}(\rho\ a^{ij}\partial_{x_{j}} \varphi) + c\; \varphi,\quad \varphi \in D_{0} := C_{0}^{\infty}(\R^{d}),
\end{align}
where $a^{ij}, 1 \leq i, j \leq d$, and $c$ satisfy assumption (A) below, which we assume to hold throughout this section:

\begin{enumerate}
\item[(A)]
$a^{ij}, c : \R^{d} \to \R$ are $\BC(\R^{d})$ measurable, $c \leq 0$, and $A(x) := (a^{ij}(x))_{1 \leq i, j \leq d}$ is a 
nonnegative definite matrix for all $x \in \R^{d}$.
Furthermore,
$$a^{ij} \in W_{loc}^{1, 1}(\R^{d}, dx) \cap L^{2}_{loc}(\R^{d}, \rho\; dx); c, \partial_{x_{i}} a^{ij} \in L^{2}_{loc}(\R^{d}, \rho\; dx), \\
\rho^{\frac{1}{2}} \in W_{loc}^{1, 1}(\R^{d}, dx)
$$
 such that
 $$a^{ij} \rho^{-\frac{1}{2}} \partial_{x_{i}} \rho^{\frac{1}{2}} \in L_{loc}^{2}(\R^{d}, \rho\; dx)
 $$
  for all $1 \leq i, j \leq d.$
\end{enumerate}

\begin{Remark}\label{Remark4.1}
{\rm
We note that (A) is a standard a priori assumption on $L_{0}$ in \eqref{eq:4.1}, because it implies the following:
\item[(i)] for every $\varphi \in C_{0}^{\infty}(\R^{d})$
\begin{align}\label{eq:4.2}
L_{0} \varphi = a^{ij} \partial_{x_{i}} \partial_{x_{j}} \varphi + (\partial_{x_{i}}a^{ij}) \partial_{x_{j}} \varphi 
+ 2 \rho^{-\frac{1}{2}} \partial_{x_{i}} \rho^{\frac{1}{2}} a^{ij} \partial_{x_{j}} \varphi + c \varphi,
\end{align}
and $(L_{0}, C_{0}^{\infty}(\R^{d}))$ is symmetric on $L^{2}(\R^{d}, \rho \; dx)$, i.e., $L_{0} \subset L_{0}^{*}$,
 where the adjoint is taken in $L^{2}(\R^{d}, \rho \; dx)$.

\item[(ii)] The nonnegative definite  symmetric bilinear from
\begin{align*}
\mathcal{E}_{0}(\psi, \varphi) :&= - \int \psi\ L_{0}\ \hspace{-0.1em}\varphi\ \rho \;  dx \\
					&= \int \langle A \nabla \psi, \nabla \varphi \rangle_{\R^{d}}\ 							 
\rho \; dx - \int c\ \psi\ \varphi\ \rho\ dx; \;\; \psi, \varphi 							 \in C_{0}^{\infty}(\R^{d}),
\end{align*}
is a symmetric pre-Dirichlet form, hence its closure $(\mathcal{E}_{F}, D(\mathcal{E}_{F}))$ is a symmetric Dirichlet form, 
whose corresponding generator $(-L_{F}, D(L_{F}))$ is just the Friedrichs extension of $(L_{0}, C_{0}^{\infty}(\R^{d}))$. 
In particular, $T_{t}^{L_{F}} := e^{tL_{F}}, t \geq 0$, is sub-Markovian. We refer to \cite[Section 3.3]{FOT2011} a
nd \cite[Chapter II, Section 1a) and 1c)]{MR1992} for details on the standard proofs for the above claims. 
In particular, for $(L_{0}, D_{0})$ as above
\begin{align*}
\MC_{sym, M}(L_{0}, D_{0}) \neq \emptyset.
\end{align*}
}\end{Remark}

Below we shall present various sets of additional assumptions on $a^{ij}, 1 \leq i, j \leq d$, and $c$ 
so that a respective theorem from the previous section will apply to imply
\begin{align*}
\#\; \MC_{sym, M}(L_{0}, D_{0}) = 1,
\end{align*}
i.e., to imply Markov uniqueness for $(L_0, C_{0}^{\infty}(\R^d))$ on $L^{2}(\R^d, \rho\ dx)$. 
We briefly repeat the set-up in each subsection to ease selective reading.

\begin{Remark}\label{Remark4.2}
\rm
As mentioned above, we only consider time-independent coefficients for the operator in \eqref{eq:4.1} 
and assume symmetry of $L_0$ on some weighted $L^2$-space over $\R^d$. As shown in Section 2, 
however, our approach is much more general and could be applied also to non-symmetric cases and for more 
general state spaces than merely $E = \R^{d}$. By time-space homogenization one can also find applications 
of the theorems in Section 3 to the cases of time-dependent coefficients (and the associated generalized 
Dirichlet forms; see \cite{St1999} and \cite{T2000}. A starting point for the nonsymmetric case could 
be the case of an operator $L_0$ as in \eqref{eq:3.1} with time-independent coefficients and with $c \equiv 0$, 
which has an infinitesimally invariant measure $\mu$, or equivalently has a stationary solution $\mu$ to its 
corresponding FPKE \eqref{eq:3.2}. This case has been studied intensively in \cite{BKRS2015} in Chapters 1-5. 
In particular, it has been shown there that under broad conditions $\mu$ has a reasonably regular density with 
respect to Lebesgue measure and $L_0$ can be written as the sum of a symmetric operator $L_{sym}$ on $L^{2}(\R^{d}, \mu)$ 
and a vector field $b$ which has divergence zero with respect to $\mu$. In this case
$L_{0}^{*}$ on $L^{2}(\R^{d}, \mu)$, calculated on $D_{0}(= C_{0}^{\infty}(\R^{d}))$, 
is just given by $L_{sym}-\langle b, \nabla\rangle_{\R^{d}}$ and then one can proceed 
analogously as in the symmetric case to obtain Markov uniqueness results in this nonsymmetric case, 
which falls into the class (1) introduced in the Introduction.
\end{Remark}

\subsection{Nondegenerate VMO diffusion coefficients}

Let $(L_0, D_0)$ be as in \eqref{eq:4.1} (respectively, \eqref{eq:4.2}) and assume
that assumption (A) holds. Let
$$\MC_{sym, M} := \MC_{sym, M}(L_0, D_0)$$
be as defined in Section 1.

\begin{Theorem}\label{thm4.3}
Suppose (A) and (H1) hold and that $a^{ij} \in VMO_{x, loc}([0, T] \times \R^{d}), \\ 1 \leq i, j \leq d$. 
Additionally, assume that for $1 \leq i, j \leq d$
\begin{align}\label{eq:4.3}
a^{ij}, \partial_{x_{i}}a^{ij} + a^{ij} \rho^{-\frac{1}{2}} \partial_{x_{i}} \rho^{\frac{1}{2}}, \; 
c \in L^1(\R^d, \rho \; dx) + L^{\infty}(\R^{d}, \rho \; dx).
\end{align}
Then $$\MC_{sym, M} = \{L_{F} \},$$ \\
i.e. Markov uniqueness holds for $(L_0, C_0^\infty(\R^d))$ on $L^{2}(\R^{2}, \rho\ dx)$.
\end{Theorem}

\begin{proof}
Let $L \in \MC_{sym, M}$ and $T_{t}^{L} := e^{tL}, t \geq 0$. Let $u \in L^{\infty}(\R^{d}, \rho \; dx), u \geq 0, \\
 \int u\; \rho \; dx = 1$ and $\mu_{t}^{L} := T_{t}^{L}u\ \rho \; dx, t \geq 0$. Then $(\mu_{t}^{L})_{t \geq 0} \in \mathcal S\mathcal P$ 
 for all $t \geq 0$ and $\mu_{0} = u \rho \; dx =: \nu$. Now let us check that $(\mu_{t}^{L})_{t \geq 0} \in \mathcal{S} \mathcal{P}_{\nu}$, 
 i.e., satisfies \eqref{eq:3.3} - \eqref{eq:3.6}.
We have seen in \eqref{eq:2.9} that $(\mu_{t}^{L})_{t \geq 0}$ solves the FPKE \eqref{eq:3.2}, hence \eqref{eq:3.3} holds. \\
From \eqref{eq:4.2} it follows that $L_0$ in this section is of type \eqref{eq:3.1} with
\begin{align}\label{eq:4.4}
b^{j} := \partial_{x_{i}}a^{ij} + 2 a^{ij} \rho^{-\frac{1}{2}} \partial_{x_{i}} \rho^{\frac{1}{2}} , 1 \leq j \leq d.
\end{align}
Since $T_{t}^{L}u \in (L^{1} \cap L^{\infty})(\rho \; dx)$, it follows from (A) and condition \eqref{eq:4.3} 
that also \eqref{eq:3.4}, \eqref{eq:3.5} holds, and additionally we have that
\begin{align}\label{eq:4.5}
a^{ij}, b^{j} \in L^{1}([0, T] \times \R^{d}; \mu_{t}^{L} dt) \; 1 \leq i, j \leq d.
\end{align}
So, it remains to check the second half of \eqref{eq:3.6}. To this end let $\chi_{n} \in C_{0}^{\infty}(\R^{d}), n \in \N$, 
such that $\mathbbm{1}_{B_{n}} \leq \chi_n \leq \mathbbm{1}_{B_{n + 1}}$ for all $n \in \N, \sup_{n \in \N} \|\chi^{'}_{n}\|_{\infty},\;  
\sup_{n \in \N} \|\chi^{''}_{n}\|_{\infty} < \infty$, and $\chi_{n}\nearrow$ in $n$, where $B_n$ denotes the ball in $\R^{d}$ 
with center $0$ and radius $n$. 

Then by \eqref{eq:4.3} for all $t \geq 0$
\begin{align}\label{eq:4.6}
\nu(\R^d) - \mu_{t}^{L}(\R^d) &= \int u \rho \; dx - \lim_{n \to \infty} 										 
\int \chi_n\ T_{t}^{L}u\ \rho \; dx \notag\\
							&= \lim_{n \to \infty} \int (1 - T_{t}^{L}\chi_n) u\ 										 
\rho \; dx) \notag\\
							&= \lim_{n \to \infty} \int (1 - \chi_n - 											 
\int_{0}^{t} T_{s}^{L} L_{0} \chi_n \; ds) 													
u\ \rho \; dx \notag\\
							&= - \lim_{n \to \infty} \int_{0}^{t} \int L_{0}\ \chi_n\ 	 
T_{s}^{L}\ u\ \rho\ dx\ ds \notag\\
							&= - \int_{0}^{t} \int c\ d\mu_{s}^{L}\ ds
\end{align}
and the second part of \eqref{eq:3.6} follows even with equality sign. Hence $(\mu_{t}^{L})_{t \geq 0} \in \mathcal{S \mathcal{P}_\nu}$. 
By \eqref{eq:4.5} it thus follows that $(\mu_{t}^{L})_{t \geq 0}$ also lies in the set defined in \eqref{eq:3.7}. 
Since $T_{t}^{L}, t \geq 0$, 
is uniquely determined by its values on all functions $u$ as above 
and $L \in \MC_{sym, M}$ was arbitrary, Theorem 3.1 implies that $$\# \MC_{sym, M} \leq 1.$$ 
Now the assertion follows by Remark 4.1(ii).
\end{proof}

\subsection{Nondegenerate locally Lipschitz diffusion coefficients}

Let $(L_0, D_0)$ be as in \eqref{eq:4.1} (respectively,
 \eqref{eq:4.2}) such that assumption (A) holds and let $\MC_{sym, M} := \MC_{sym, M}(L_0, D_0)$ be defined as in Section 1.
In the following result we shall assume (H2) for our $a^{ij}, 1 \leq i,j \leq d$, which is stronger than the local VMO-condition in Theorem 4.3.
As a reward we can relax the global conditions in \eqref{eq:4.3}. We need, however, 
to restrict to the case $c \equiv 0$.

\begin{Theorem}\label{thm4.4}
Suppose that $c \equiv 0$ and that conditions (A), (H1) and (H2) hold. 
Additionally, assume that for $1 \geq i, j \geq d$, and some $p > d + 2$
\begin{align}\label{eq:4.7}
\rho^{-\frac{1}{2}} \partial_{x_{i}} \rho^{\frac{1}{2}} \in L_{loc}^{p}(\R^{d}, dx),
\end{align}
and that
\begin{align}\label{eq:4.8}
\frac{\vert a^{ij} \vert}{(1 + \vert x \vert^2)} 
+ \frac{\vert \partial_{x_{i}}a^{ij} 
+ a^{ij}\rho^{-\frac{1}{2}}\partial_{x_{i}} \rho^{\frac{1}{2}} \vert}{(1 + \vert x \vert)} \in L^{1}(\R^d, \rho\ dx) + L^{\infty}(\R^{d}, \rho\ dx)
\end{align}
Then $$\MC_{sym, M} = \{L_{F} \},$$ i.e., Markov uniqueness holds for $(L_0, C_{0}^{\infty}(\R^d))$ on $L^{2}(\R^{d}, \rho\ dx)$.
\end{Theorem}

\begin{proof}
Define $b = (b^j)_{1 \leq i \leq d}$ as in \eqref{eq:4.4}. We note that by (H2) we have $\partial_{x_{i}}a^{ij} \in L_{loc}^{\infty}(\R^d, dx)$ 
for $1 \leq i, j \leq d$. Let $L \in \MC_{sym, M}$ and let $(\mu_{t}^{L})_{t \geq 0}, \nu, \chi_n, n \in \N$, 
be as defined in the proof of Theorem 4.3. Then for every $t \geq 0$, since $T_{t}^{L}$ is sub-Markovian, we have 
\begin{align*}
\mu_{t}^{L}(\R^d) &= \int T_{t}^{L}u\ \rho\ dx \\
				&= \lim_{n \to \infty} \int \chi_n\ T_{t}^{L}u\ \rho\ dx \\
				&= \lim_{n \to \infty} \int T_{t}^{L}\ \hspace{-0.1cm}\chi_n\ u\ \rho\ dx\\
				&\leq \int u\ \rho\ dx = \nu (\R^d).
\end{align*}
Hence \eqref{eq:3.6} holds and then exactly as in the proof of Theorem 4.3 on checks (without using (4.8)) 
that by assumption (A) also \eqref{eq:3.3}--\eqref{eq:3.5} hold 
to conclude that $(\mu_{t}^{L})_{t \geq 0} \in \mathcal{S \mathcal{P}_{\nu}}$. 
Furthermore, since $T_{t}^{L}u \in (L^1 \cap L^\infty)(\R^d, \rho \; dx)$, 
the left-hand side of (4.8) is also an element of $L^1 ([0, T] \times \R^d, \mu_{t}^{L} dt)$, 
hence by \eqref{eq:4.7} all assumptions of Theorem 3.2 are fulfilled. So, $\# \MC_{sym, M} \leq 1$, and Remark 4.1(ii) implies the assertion.
\end{proof}

\begin{Remark}
\yx
\rm
Let us mention the uniqueness problem studied \cite{Hille}  for the one-dimensional
Fokker-Planck-Kolmogorov equation. For simplicity we consider the case
of the unit diffusion coefficient (note that in \cite{Hille} the opposite notation is used, the drift
is denoted by $a$, but we follow our notation).
The problem posed  in \cite[\S8, p.~116]{Hille}   (in the case of the equation on the whole real line)
is this: to find necessary and sufficient conditions in order that
for every function $h\in L^1(\mathbb{R})$ with $Lh=h''-(bh)'\in L^1(\mathbb{R})$
there is a unique solution  $T(x,t,h)$ of the equation  $\partial_t u=\partial_x^2 u-\partial_x(ub)$
with initial condition $h$ in the sense of the relation $\|T(\cdot,t,h)-h\|_{L^1}\to 0$ as $t\to 0$.
 This setting is called Problem~$L_0$, and in Problem~$L$ it is required  in addition that the
 solutions with probability initial densities from the domain of definition
 of the operator $L$ must be  probabilistic. 
 According to \cite[Theorems~8.5 and~8.7]{Hille}, where the drift coefficient  is assumed
 to be continuous,
 a necessary and sufficient condition for the solvability of Problem $L_0$ is the divergence of the  integral
 $$
 \int_0^x \exp B(y)\int_0^y \exp (-B(u))\, du \, dy, \quad \hbox{where} \quad B(y)=\int_0^y b(s)\, ds
 $$
at $-\infty$ and $+\infty$, and  for the solvability of Problem $L$
the divergence of the integral
$$
\int_0^x \exp (-B(y))\int_0^y \exp (B(u))\, du \, dy
$$
at $-\infty$ and $+\infty$ is additionally required. This is the previous  condition for the drift $-b$,
which makes the conditions for $b$ and $-b$ the same.
In both cited theorems of Hille the closure of the operator $L$ generates a semigroup on $L^1(\mathbb{R})$.
It is proved in \cite{BKS21} that a probability solution is always unique in the one-dimensional case 
(under the stated assumptions about $a$ and~$b$). 
However, an example constructed in \cite{BKS21} shows that the situation is possible where for an
initial condition that is a probability measure there exists a unique probability solution of the Cauchy problem,
but there are also other solutions. 
It is worth noting that it is asserted in Remark 4.6 in \cite{BKS21} that if Hille's condition is violated, 
then for some initial condition there is no solution, but this does follow from the results in \cite{BKS21}, because 
they ensure uniqueness only for probability solutions, so that one cannot rule out the possibility that 
existence holds for all initial solutions, but uniqueness fails in the class of signed solutions. 
\xy
\end{Remark}

\subsection{Nondegenerate diffusion coefficients and Lyapunov function conditions}

Let $(L_0, D_0)$ be as in \eqref{eq:4.1} (respectively, \eqref{eq:4.2}) 
and assume that assumption (A) holds. Let $\MC_{sym, M} := \MC_{sym, M}(L_0, D_0)$ be as defined in Section 1.

\begin{Theorem}\label{thm4.5}
Suppose that $c \equiv 0$ and that conditions (A), (H1) and (H2) hold. Additionally, assume 
that \eqref{eq:4.7} holds and that \eqref{eq: 3.8} holds with $b = (b^{j})_{1 \leq j \leq d}$ defined as in \eqref{eq:4.4}. 
Then $$\MC_{sym, M} = \{ L_F \},$$ i.e., Markov uniqueness holds for $(L_0, C_{0}^{\infty}(\R^d))$ on $L^{2}(\R^{d}, \rho\ dx)$.
\end{Theorem}

\begin{proof}
The proof is completely analogous to the proof of Theorem 4.4 except for applying Theorem 3.3 and Example 3.4 
instead of Theorem 3.2 and replacing condition (4.8) by \eqref{eq: 3.8}.
\end{proof}

\begin{Remark}\label{Remark4.6}
\rm
We would like to point out that Theorem 4.6 is close to Corollary 2.3 in \cite{St1999a} and to Proposition 2.9.4 
in \cite{BRSt2000}. However, it is not covered by them, since $\rho$ is not a probability density. 
The function $\rho$ is not even assumed to be in $L^1(\R^d, dx)$ here.
\end{Remark}

\subsection{Degenerate diffusion coefficients}

\subsubsection{Markov uniqueness as a consequence of the results of Le Bris and Lions}

\begin{Theorem}\label{thm4.7}
Let $\sigma := (\sigma^{ij})_{1 \leq i, j \leq d}, A := \sigma \sigma^{*}$ and $a^{ij} = (\sigma \sigma^{*})^{ij}, 1 \leq i, j \leq d$, where
\begin{align}\label{eq:4.9}
\sigma^{ij} \in W_{loc}^{1, 2}(\R^d, dx), \partial_{x_{i}}a^{ij} \in W_{loc}^{1, 2}(\R^d, dx)
\end{align}
and
\begin{align}\label{4.10}
\sigma^{ij}, \partial_{x_{j}} \partial_{x_{i}}a^{ij} \in L^\infty(\R^d, dx), 
\frac{\partial_{x_{i}}a^{ij}}{1 + \vert x \vert} \in L^1(\R^d, dx) + L^\infty(\R^d, dx).
\end{align}
Then condition (A) holds for $\rho \equiv 1, c \equiv 0$, and the corresponding operator $(L_0. D_0)$ 
from \eqref{eq:4.2} is symmetric on $L^2(\R^d, dx).$ Let $\MC_{sym, M} := \MC_{sym, M}(L_0, D_0)$ 
be as defined in Section 1. Then 
$$\MC_{sym, M} = \{L_F \},$$ 
i.e., Markov uniqueness holds for $(L_0, C_{0}^{\infty}(\R^d))$ on $L^2(\R^d, dx)$.
\end{Theorem}

\begin{proof}
Let $L \in \MC_{sym, M}$ and $\mu_{t}^{L} := T_{t}^{L}u \; dx, t \geq 0$, with $u$ as in the proof of Theorem 4.3. 
Then by assumptions (4.9), (4.10), we can apply Theorem 3.5 with $\rho_0 := u$, 
since $T_{t}^{L}u \in (L^1 \cap L^\infty)(\R^d, dx)$ and $\sigma^{*} \nabla T_{t}^{L}u \in L^2(\R^d, dx;\ \R^d)$, 
because $\nabla T_{t}^{L}u \in L^2(\R^d, dx;\ \R^d)$ and $\sigma^{ij} \in L^\infty (\R^d, dx), 1 \leq i, j\leq d$. 
Hence $\# \MC_{sym, M} \leq 1$ and by Remark 4.1(ii) the assertion follows.
\end{proof}

\subsubsection{Markov uniqueness in another degenerate case}

Let $\rho \in (L^1 \cap L^3)(\R^d, dx)$ such that 
$$\rho > 0, \ \int \rho dx = 1, \ \rho^{\frac{1}{2}} \in W_{loc}^{1, 1}(\R^d, dx)
$$
 and 
$\nabla \rho^{\frac{1}{2}} \in L^\infty(\R^d, dx;\ \R^d)$, and assume that (H3) holds. Consider the operator
\begin{align}
L_0 \varphi := \rho\; div (A \nabla \varphi) + \sqrt{\rho} \langle A \nabla \sqrt{\rho}, 
\nabla \varphi\rangle_{\R^d}, \varphi \in D_0 := C_{0}^{\infty}(\R^d),
\end{align}
and its corresponding FPKE (3.2). Note that by our assumptions on $A$ and $\rho$ we have that 
$L_0\colon D_0 \subset L^2(\R^d, \rho \; dx) \to L^2(\R^d, \rho \; dx)$ and 
$L_0 \varphi = \frac{1}{\rho} {\rm div}(\rho^2 A \nabla \varphi)$ for all $\varphi \in D_0 = C_{0}^{\infty}(\R^d)$, 
hence $(L_0, D_0)$ is symmetric on $L^2(\R^d, \rho \; dx)$.
Let $\MC_{sym, M} := \MC_{sym, M}(L_0, D_0)$ be as defined in Section 1.

\begin{Theorem}\label{thm4.8}
Assume that (H3) holds and let $\rho$ satisfy the assumptions specified above.
Then $$\MC_{sym, M} = \{L_F \},$$ i.e., Markov uniqueness holds for $(L_0, C_{0}^{\infty}(\R^d))$ on $L^{2}(\R^{d}, \rho\ dx)$.
\end{Theorem}

\begin{proof}
Let $L \in \MC_{sym, M}$ and $\mu_{t}^{L} := T_{t}^{L}u\ \rho \; dx, t \geq 0$, with $u$ as in the proof 
of Theorem 4.3. We have seen in \eqref{eq:2.9} that $(\mu_{t}^{L})_{t \geq 0}$ solves the FPKE associated 
with $(L_0, D_0)$ in (4.11). To show that it is the only such solution we are going to apply Theorem 3.6. So, 
let us check its assumptions for $z(t, \cdot) := T_{t}^{L}u\ \rho$ and $\tilde{\rho} := \rho$. 
First of all, (3.11) holds as just seen. So, let us show (3.12). 
As in \eqref{eq:4.6} we have for every $t \geq 0$
$$\mu_{t}^{L}(\R^d) = \int u \rho \; dx + \lim_{n \to \infty} \int_{0}^{t} \int L_0 \chi_n\ T_{s}^{L}u\ \rho \; dxds.$$
By our assumptions about $A$ and since $\nabla \sqrt{\rho} \in L^\infty (\R^d, dx; \; \R^{d})$, 
we have that for some $C \in (0, \infty)$ and all $s \geq 0$
\begin{align*}
\sup_{n} \vert L_0 \chi_n\ T_{s}^{L}u \vert \leq C \Vert u \Vert_{\infty} (\rho + 1), \quad dx-\hbox{a.e.}
\end{align*}
Since $\rho \in (L^1 \cap L^2)(\R^d, dx)$ and $L_0\ \hspace{-0.1cm}\chi_n \to 0\ dx$-a.e. as $n \to \infty$, we conclude 
that $$\mu_{t}^{L}(\R^d) = \int u \rho \; dx = 1 \text{\; for all } t \geq 0.$$ 
Next, (3.13) is clear, since $T_{t}^{L}u \in L^\infty (\R^d, dx)$ and $\rho \in L^{\infty}_{loc}(\R^d, dx)$, 
because $\nabla \rho^{{\frac{1}{2}}} \in L^{\infty}(\R^d, dx)$. 

Finally, let us show (3.14). 
It suffices to show that all functions under the integral in (3.14) are in $L^1 (\R^d, dx)$ in our case, 
due to our assumptions. For the first summand this is immediate, since
\begin{align*}
(\rho^{\frac{1}{2}} + \rho)\ z (t, \cdot) &= (\rho^{\frac{1}{2}} + \rho)\rho \; T_{t}^{L}u \\
		&\leq (1 + 2\rho) \rho \Vert u \Vert_\infty \in L^1(\R^d, 		dx),
\end{align*}
since $\rho \in (L^1 \cap L^2)(\R^d, dx)$ by assumption. For the second summand we note that $\rho^{\frac{1}{2}}$ has a Lipschitz $dx$-version 
on $\R^d$, since $\nabla \sqrt{\rho} \in L^\infty (\R^d, dx;\ \R^d)$ by assumption. 
Hence $\rho^{\frac{1}{2}}$ is of at most linear growth and thus $\rho$ of at most quadratic growth. 
Hence, since $\rho \in L^{3}(\R^d, dx)$, for some $C \in (0, \infty)$ and all $t \geq 0$ we have
\begin{align*}
\frac{\rho^2 (x)}{1 + \vert x \vert^2}\; z^{2}(t, \cdot) \leq C \rho^3 \Vert u \Vert_\infty \in L^1(\R^d, dx),
\end{align*}
and altogether (3.14) follows. Since (H4) also holds by our assumptions about $A$ 
and $\nabla \sqrt{\rho} \in L^\infty (\R^d, dx;\ \R^d)$, we can apply Theorem 3.6 and 
conclude that $\# \MC_{sym, M} \leq 1$ and again by Remark 4.1(ii) the assertion follows.
\end{proof}

\yx
\section*{Acknowledgements}
Financial support by the HCM in Bonn, by the Deutsche Forschungsgemeinschaft (DFG, German Research Foundation)- SFB 1283/2 2021 - 317210226, the Russian Foundation for Fundamental Research Grant 20-01-00432, Moscow Center of Fundamental Applied Mathematics, and the Simons-IUM fellowship are gratefully acknowledged.
\xy

\end{document}